\documentclass[12pt,reqno,a4paper]{amsart}
\usepackage{blindtext}
\usepackage{fullpage}
\usepackage{mathtools}
\usepackage{longtable}
\usepackage{amsmath,amssymb,amsthm}
\usepackage{amscd}
\usepackage{bm}
\usepackage{hyperref}   
\usepackage{dsfont}
\usepackage{enumerate}
\usepackage{epsfig}
\usepackage{float, graphicx}
\usepackage{latexsym, amsxtra}
\usepackage{mathrsfs}
\usepackage{multicol}
\usepackage[normalem]{ulem}
\usepackage{psfrag}
\usepackage[parfill]{parskip}
\usepackage{stmaryrd}
\usepackage{tikz}
\usepackage[T1]{fontenc}
\usepackage{url}
\usepackage{verbatim}
\usepackage{indentfirst}
\usepackage{tikz-cd}
\usepackage{booktabs}
\usepackage{svg}

\usepackage{caption} 

\makeatletter
\def\thm@space@setup{%
	\thm@preskip=2ex \thm@postskip=2ex
}
\makeatother

\hypersetup{hidelinks}

\newtheorem{thm}{Theorem~}[section]
\newtheorem{lem}[thm]{Lemma~}

\newtheorem{prop}[thm]{Proposition~}
\newtheorem{ques}[thm]{Question~}
\newtheorem{cor}[thm]{Corollary~}

\newtheorem{conj}[thm]{Conjecture~}

\theoremstyle{remark}
\newtheorem{rmk}[thm]{Remark~}

\theoremstyle{definition}  

\newcommand{\CC}{\mathbb{C}}
\newcommand{\ZZ}{\mathbb{Z}}

\newcommand{\PP}{\mathbb{P}}

\newcommand{\calA}{\mathcal{A}}
\newcommand{\calB}{\mathcal{B}}

\newcommand{\calO}{\mathcal{O}}

\newcommand{\calS}{\mathcal{S}}
\newcommand{\calI}{\mathcal{I}}

\newcommand{\calN}{\mathcal{N}}

\newcommand{\GL}{\mathrm{GL}}

\newcommand\Sym{\mathrm{Sym}}

\newcommand\Spec{\mathrm{Spec}}

\newcommand{\rank}{\mathrm{rank}}
\newcommand{\Pic}{\mathrm{Pic}}

\newcommand{\Gr}{\mathrm{Gr}}

\newcommand{\pr}{\mathrm{pr}}

\newcommand{\calQ}{\mathcal{Q}}
\newcommand{\calW}{\mathcal{W}}
\newcommand{\calHom}{\mathcal{H}om}

\newcommand{\sm}{\mathrm{sm}}
\newcommand{\red}{\mathrm{red}}
\newcommand{\codim}{\mathrm{codim}}
\newcommand{\NF}{\mathrm{NF}}
\newcommand{\Frac}{\mathrm{Frac}}
\newcommand{\PGL}{\mathrm{PGL}}
\newcommand{\Sing}{\mathrm{Sing}}

\hypersetup{
 pdftitle={Dimensions of Fano schemes of lines on hypersurfaces},
 pdfauthor={Chenglong Yu}
}
\title{Dimensions of Fano schemes of lines on hypersurfaces}
\vspace{1.2cm}
\author{Chenglong Yu}
\date{}

\newcommand{\Addresses}{{%
 \bigskip
 \footnotesize
 C.~Yu, \textsc{Center for Mathematics and Interdisciplinary Sciences, Fudan University and Shanghai Institute for Mathematics and Interdisciplinary Sciences (SIMIS), Shanghai, China}\par\nopagebreak
 \textit{Email address}: \texttt{yuchenglong@simis.cn}
}}

\begin{document}
\bibliographystyle{amsalpha-shorthand}
\begin{abstract}
Let $X\subset\PP^n_\CC$ be a smooth hypersurface of degree $d\ge2$,
with $n\ge3$. We prove that an $m$-dimensional integral projective family
of lines sweeping a variety of dimension $y$ satisfies
$m\le\max\{2y-d-1,n-3\}$. This proves the de Jong--Debarre
conjecture for $d\le n$ and the Beheshti--Starr bound for lines on
smooth hypersurfaces for $d\ge n+1$.
\end{abstract}
\maketitle

\section{Introduction}\label{sec:intro}
In this paper, we assume all schemes are of finite type over $\CC$ and a variety means an integral separated scheme. Let $X\subset\PP^n_\CC$ be a smooth hypersurface of degree $d\ge2$,
with $n\ge3$. Its Fano scheme of lines $F_1(X)\subset\Gr(2,n+1)$
parametrizes lines contained in $X$. It is locally cut out by $d+1$
equations, so every nonempty component has dimension at least
$2n-d-3$. In the Fano range, the expected upper bound is the following.

\begin{conj}[de Jong--Debarre {\cite[Conjecture~1]{Debarre03}}]
\label{conj:DJD}
If $d\le n$, then $\dim F_1(X)=2n-d-3$.
\end{conj}

For higher degrees, the Beheshti--Starr question
\cite[Question~1.3]{B05}, restricted to lines on smooth
hypersurfaces, takes the following form \cite[p.~574]{LT}.
\begin{ques}[Beheshti--Starr]\label{ques:BS}
If $d\ge n$, does every irreducible component of $F_1(X)$ have
dimension at most $n-3$?
\end{ques}
Here $F_1(X)$ may be empty. For every $n\ge3$ and $d\ge n+1$,
the classical Fermat example attains the bound $n-3$
\cite[p.~1]{Debarre03}. See \S\ref{subsec:fermat}. A stronger statement allowing
singular hypersurfaces is formulated in \cite[Conjecture~1.1]{LT}.

For $d\le n$, Beheshti proves Conjecture~\ref{conj:DJD} in degree
at most eight \cite{B06,B14}.
Beheshti--Riedl prove it for $n\ge2d-4$ using equations for lines
through a fixed point and a bound on the dimension of a tangency
locus \cite[Corollary~2.5]{BR}. They also prove
irreducibility for $n\ge2d-1$ and $n\ge4$
\cite[Corollary~2.8]{BR}. Landsberg--Tommasi relate both problems to Chern-class nonvanishing
on families of lines through a point \cite{LT}.
We prove the following result.

\begin{thm}\label{thm:main}
Let $X\subset\PP^n_\CC$ be a smooth hypersurface of degree $d\ge2$, with $n\ge3$.
\begin{enumerate}[(i)]
\item If $d\le n$, then $F_1(X)$ is nonempty and every irreducible
component has dimension $2n-d-3$. 
\item If $d\ge n+1$, every irreducible component of $F_1(X)$ has
dimension at most $n-3$.
\end{enumerate}
\end{thm}
This proves Conjecture~\ref{conj:DJD} and answers
Question~\ref{ques:BS}. The main estimate comes from the dimension of the swept variety.

\begin{thm}\label{thm:swept}
Let $S\subset F_1(X)$ be a closed subvariety of dimension
$m$, and let $y$ be the dimension of the union of its lines. Then
\begin{equation}\label{eq:swept}
 m\le\max\{2y-d-1,\,n-3\}.
\end{equation}
\end{thm}

The swept-variety estimate also gives the following consequences.
A line $\ell\subset X$ is \emph{free} if its normal bundle
$N_{\ell/X}$ is globally generated.
\begin{cor}\label{cor:free-irreducible}
Let $X\subset\PP^n_\CC$ be a smooth hypersurface of degree $d\ge2$, with $n\ge3$.
\begin{enumerate}[(i)]
\item If $d\le n-1$, a general line in every irreducible component
of $F_1(X)$ is free.
\item If $d\le n-2$, then $F_1(X)$ is irreducible.
\end{enumerate}
\end{cor}

For $d\ge n-2$, every closed subvariety of $F_1(X)$ of dimension
at least $n-2$ parametrizes lines covering $X$
(Corollary~\ref{cor:debarre-covering}). This proves Debarre's
covering conjecture \cite[Conjecture~2]{Debarre03}.
The Fano scheme is also a local complete intersection for $d\le n$,
reduced for $d\le n-1$, and normal and integral for $d\le n-2$
(Corollary~\ref{cor:scheme-structure}).

We combine the deformation theory of lines through a fixed point,
as in Beheshti--Riedl, with the Chern-class approach of
Landsberg--Tommasi. For a family of dimension at least $n-2$, we
compare two bounds for the degree of the tautological quotient
on the Nash modification along a lifted general line.
Deformation theory gives the upper bound. The new step compares
determinants up to an effective divisor on a normalized graph
closure. The gradient map and Schubert calculus then give the
lower bound $d-1$ for the same degree. This allows singular swept
varieties and avoids the tangency-locus dimension bound of
\cite[Lemma~2.1]{BR}.

The paper is organized as follows.
In \S\ref{sec:nash}--\S\ref{sec:pointed}, we establish the line
estimate, the numerical quotient inequality, and the determinant
comparison. We prove the dimension theorems and
Corollaries~\ref{cor:debarre-covering} and~\ref{cor:free-irreducible}
in \S\ref{sec:proof}. In \S\ref{sec:comparison}, we compare our
argument with the earlier methods, derive bounds for nonfree lines,
and prove Corollary~\ref{cor:scheme-structure}.

\section{The Quotient Bundle on the Nash Modification}\label{sec:nash}
We construct the tautological quotient bundle on the Nash
modification of the swept variety and bound its degree along
the lift of a line by the dimension of the family.

We use the convention that $\PP(E)$ parametrizes one-dimensional
subspaces, with $\calO_{\PP(E)}(-1)\subset p^*E$.
A subbundle has locally free quotient. Degrees of vector bundles
on curves mean degrees of determinant bundles or equivalently the degrees of first Chern classes. We use Chern classes as
operations on Chow groups and integrate against fundamental cycles
\cite{Fulton}.

Let $\calS$ be the tautological bundle on $\Gr(2,n+1)$.
Restriction of an equation of $X$ realizes $F_1(X)$ as the zero
scheme of a section of $\Sym^d\calS^\vee$.
For a closed subvariety $S\subset F_1(X)$, let
$p_S\colon \calW=\PP_S(\calS|_S)\to S$ be its universal family,
with evaluation $e\colon \calW\to X$. This is the incidence variety
in $S\times X$ defined by $x\in\ell$, with its two projections. Put
\[
 m=\dim S,\qquad Y=e(\calW)_{\red},\qquad
 y=\dim Y,\qquad r=n-y,\qquad t=m+1-y.
\]
The projective bundle $\calW$ is integral, so its reduced
image $Y$ is integral. 

Write $V=\CC^{n+1}$. For $x\in Y_{\sm}$, let
$\widehat T_{Y,x}\subset V$ be the $(y+1)$-dimensional vector
space whose projectivization is the embedded tangent space.
Let $\widehat Y$ be the reduced graph closure of
$x\mapsto\widehat T_{Y,x}$ in $Y\times\Gr(y+1,V)$, with
projection $\nu\colon \widehat Y\to Y$. This is the projective Gauss-map
realization of the classical Nash modification
\cite[pp.~297--298]{Nobile}. The tautological sequence is
\begin{equation}\label{eq:nash}
 0\longrightarrow\calA_Y\longrightarrow V\otimes\calO_{\widehat Y}
 \longrightarrow\calB_Y\longrightarrow0,
 \qquad \rank\calA_Y=y+1,\quad\rank\calB_Y=r.
\end{equation}
Both terms are vector bundles, and $\calB_Y$ is globally generated.

The partial derivatives of a defining form $F$ of $X$ give the
gradient map $V\otimes\calO_Y\to\calO_Y(d-1)$, surjective since $X$ is smooth.
It annihilates $\widehat T_{Y,x}$ because differentiation of $F|_Y=0$
gives vanishing on the tangent directions, while Euler's identity
gives vanishing on the radial direction. After pullback, its restriction to $\calA_Y$
is zero on a dense open set, hence zero on the integral variety
$\widehat Y$. Thus it induces an everywhere-surjective map
\begin{equation}\label{eq:gradient}
 \calB_Y\twoheadrightarrow\nu^*\calO_Y(d-1).
\end{equation}

\begin{prop}\label{prop:line}
There is a nonempty open $S^\circ\subset S_{\sm}$ such that every
$\ell=\ell_s$, $s\in S^\circ$, has a unique lift
$\lambda_\ell\colon \ell\to\widehat Y$, defined on the whole line, with
\begin{equation}\label{eq:line-bound}
 0\le b_\ell:=\deg\lambda_\ell^*\calB_Y\le2y-m-2.
\end{equation}
\end{prop}
\begin{proof}
By generic smoothness, $e$ maps a nonempty open
$\Omega\subset p_S^{-1}(S_{\sm})$ to $Y_{\sm}$ with differential
of rank $y$ \cite[III, \S10]{Hartshorne}. Since $p_S$ is open,
$S^\circ=p_S(\Omega)$ is nonempty and open. Each line parametrized
by $S^\circ$ meets $Y_{\sm}$, where $\nu$ is an isomorphism.
As $\ell$ is smooth, properness of $\nu$ extends its lift to the
whole line, and separatedness gives uniqueness
(Lemma~\ref{lem:extension}). Its image is the strict transform
of $\ell$, with no additional components over singular points.

Let $A\subset V$ be the two-dimensional vector subspace defining $\ell$, and put
$\calA_\ell=\lambda_\ell^*\calA_Y$ and
$\calB_\ell=\lambda_\ell^*\calB_Y$. At a smooth point of $Y$
on $\ell$, its projective tangent space contains the entire line.
Hence the map $A\otimes\calO_\ell\to\calB_\ell$ is zero
generically and therefore everywhere. The resulting inclusion
$A\otimes\calO_\ell\subset\calA_\ell$ has constant fiber
rank two, so its quotient is locally free. Define
\[
 M_\ell=(\calA_\ell/(A\otimes\calO_\ell))\otimes\calO_\ell(1).
\]
Quotienting \eqref{eq:nash} by $A\otimes\calO_\ell$ gives an exact
sequence of bundles on the complete line
\begin{equation}\label{eq:line-sequence}
 0\longrightarrow M_\ell\longrightarrow N_{\ell/\PP^n}
 \longrightarrow\calB_\ell(1)\longrightarrow0,
 \qquad N_{\ell/\PP^n}=(V/A)\otimes\calO_\ell(1).
\end{equation}
In particular $\rank M_\ell=y-1$ and
$\deg M_\ell=y-1-b_\ell$.

The differential of $S\hookrightarrow\Gr(2,n+1)$ identifies $T_sS$ with
a subspace of $H^0(N_{\ell/\PP^n})$. The characteristic map
evaluates these sections. For $x\in\Omega\cap\ell$, quotienting
the surjective differential $de$ by the fiber direction $T_x\ell$
gives $T_sS\twoheadrightarrow T_xY/T_x\ell=(M_\ell)_x$.
The composite with the locally free quotient in
\eqref{eq:line-sequence} is therefore generically zero, hence
everywhere zero. Thus $T_sS$ injects into $H^0(M_\ell)$ and
generates $M_\ell$ generically.

By Grothendieck's splitting theorem
\cite[Theorem~2.1]{Grothendieck57}, we can write
$M_\ell=\bigoplus_j\calO_\ell(a_j)$. If $a_j<0$, then
$H^0(\calO_\ell(a_j))=0$, so every global section of $M_\ell$
has zero projection to this summand, contradicting generic
generation. Hence $a_j\ge0$, and $M_\ell$ is generated everywhere
by $H^0(M_\ell)$, although the subspace $T_sS$ need not generate
at every point. The inclusion into $\calO_\ell(1)^{n-1}$ also
excludes $a_j>1$. Therefore
\begin{equation}\label{eq:M-splitting}
 M_\ell=\calO_\ell(1)^{y-1-b_\ell}\oplus\calO_\ell^{b_\ell},
 \qquad m\le h^0(M_\ell)=2y-2-b_\ell,
\end{equation}
which proves \eqref{eq:line-bound}. If $y=1$, $M_\ell=0$ and the
same injection gives $m=b_\ell=0$.
\end{proof}
\begin{rmk}
For a line $\ell$ meeting $Y_{\sm}$, the normal sheaf
$\calN_{\ell/Y}=\calHom(\calI_{\ell/Y}/\calI_{\ell/Y}^2,\calO_\ell)$
satisfies
\[
 \calN_{\ell/Y}=M_\ell,
 \qquad T_{[\ell]}F_1(Y)=H^0(M_\ell).
\]
Dualizing the conormal sequence of $\ell\subset Y\subset\PP^n$
shows that $\calN_{\ell/Y}$ is saturated in $N_{\ell/\PP^n}$.
It agrees with the saturated subbundle $M_\ell$ over
$\ell\cap Y_{\sm}$, hence everywhere. The tangent-space identity
is Grothendieck's formula \cite[Corollary~5.3]{FGAHilbert}
(see also \cite[Proposition~2.16]{AK}). For lines as in
Proposition~\ref{prop:line}, \eqref{eq:M-splitting} gives
$H^1(\calN_{\ell/Y})=0$. If $\ell\subset Y$ is not a regular
embedding, this vanishing alone does not imply that $F_1(Y)$ is
smooth at $[\ell]$.
\end{rmk}

\section{Schubert Calculus and a Quotient Inequality}\label{sec:quotient}
We prove a numerical inequality for a globally generated bundle
with a line-bundle quotient. The proof uses Chern operations on
Chow homology, their compatibility with proper pushforward, and
Whitney's formula over an arbitrary base
\cite[Chapter~3]{Fulton}. See also \cite[Tags \href{https://stacks.math.columbia.edu/tag/02U4}{02U4} and \href{https://stacks.math.columbia.edu/tag/02UI}{02UI}]{Stacks}.

\subsection{Nonnegativity on Singular Cycles}
We record a standard consequence of Schubert calculus and
Kleiman's transversality theorem.
\begin{lem}\label{lem:schur}
Let $E$ be a globally generated vector bundle of rank $r\ge1$ on a
projective scheme $Q$, and let $L$ be a globally generated line bundle on $Q$.
For an effective $q$-cycle $\alpha$ on $Q$, one has
\begin{align*}
 \deg\bigl(c_i(E)c_1(L)^{q-i}\cap\alpha\bigr)&\ge0
 &&(0\le i\le\min\{r,q\}),\\
 \deg\bigl((c_1(E)c_i(E)-c_{i+1}(E))c_1(L)^{q-i-1}
 \cap\alpha\bigr)&\ge0
 &&(1\le i\le r,\ i+1\le q).
\end{align*}
Here $c_{r+1}(E)=0$. Neither smoothness nor normality of the
cycle supports, nor reducedness of $Q$, is assumed.
\end{lem}
\begin{proof}
Choose a generating map $\CC^N\otimes\calO_Q\twoheadrightarrow E$.
It defines $f\colon Q\to\Gr(N-r,N)$ with $E=f^*\calQ$,
where $\calQ$ is the universal rank-$r$ quotient.
By adjoining zero sections, we may assume that $s=N-r\ge2$. A generating space for
$L$ similarly defines $\psi\colon Q\to\PP^M$ with
$L=\psi^*\calO_{\PP^M}(1)$. Write $\Phi=(f,\psi)$.

We recall the convention for Schubert classes used here. For a complete flag $F_\bullet$
in $\CC^N$ and a partition $\lambda$ with at most $s$ parts, each
at most $r$, let
\[
 \Omega_\lambda=
 \{K\in\Gr(s,N):\dim(K\cap F_{r+j-\lambda_j})\ge j
 \text{ for }1\le j\le s\}.
\]
Its codimension is $|\lambda|$, and its effective fundamental
class is denoted by $\sigma_\lambda$. The special Schubert and
Giambelli formulas give
\begin{equation}\label{eq:schubert}
 c_i(\calQ)=\sigma_{(i)},\qquad
 c_1(\calQ)c_i(\calQ)-c_{i+1}(\calQ)
 =\sigma_{(i,1)}.
\end{equation}
The second equality is the determinant
$\left|\begin{smallmatrix}c_i&c_{i+1}\\1&c_1\end{smallmatrix}\right|$.
Equivalently, it is the special case
$\sigma_{(1)}\sigma_{(i)}=\sigma_{(i+1)}+\sigma_{(i,1)}$
of Pieri's formula, with the first term zero when $i=r$.
These are identities in the Chow ring of the Grassmannian
\cite[\S14.7]{Fulton}, \cite[p.~36]{FL}.
In Chern-root notation the partitions are transposed, so
$c_i=e_i=s_{(1^i)}$ and the second class is $s_{(2,1^{i-1})}$.
Thus both classes in \eqref{eq:schubert} have effective
Schubert representatives.

For either class $\sigma$ in \eqref{eq:schubert}, of codimension
$j\le q$, put $H_{\PP}=c_1(\calO_{\PP^M}(1))$. The projection
formula gives
\begin{equation}\label{eq:ambient-positivity}
 \deg\bigl(c_1(L)^{q-j}f^*\sigma\cap\alpha\bigr)
 =\deg\bigl(\pr_{\Gr(N-r,N)}^*\sigma\,
 \pr_{\PP}^*H_{\PP}^{q-j}\cap\Phi_*\alpha\bigr).
\end{equation}
Here $f^*\sigma$ denotes the corresponding operational Chern
polynomial. No flatness of $f$ is asserted.
The proper pushforward $\beta=\Phi_*\alpha$ is effective or zero.
A component whose image has smaller dimension contributes zero.
Otherwise, it contributes its image with positive function-field degree.
If $q-j>M$, the right side is zero. Otherwise its class is
represented by the product of a Schubert variety and a linear
subspace of codimension $q-j$.

The group $\GL_N\times\PGL_{M+1}$
acts transitively on $\Gr(N-r,N)\times\PP^M$. A general translate of
this product meets every component of $\beta$ properly in
dimension zero. It avoids the singular loci of both factors of
the intersection, since their expected intersection dimensions
are negative. In characteristic zero the remaining intersections
are transverse on the smooth loci
\cite[Theorem~2 and Corollary~8]{Kleiman}.
The resulting zero-cycle has nonnegative multiplicities.
Changing the flag and the linear subspace preserves their Chow
classes, so \eqref{eq:ambient-positivity} is nonnegative.
The case $i=0$ follows from the same argument with the unit
Schubert class. For nonreduced schemes, we apply the argument
to the integral cycle supports with their positive
multiplicities. This proves the lemma.
\end{proof}

\subsection{The Numerical Inequality}
\begin{prop}\label{prop:numerical}
Let $W$ be a projective variety of dimension $k$, possibly singular.
Let $L$ be a globally generated line bundle and $B$ a globally generated
vector bundle on $W$ of rank $1\le r\le k$.
Suppose $D$ is an effective Cartier divisor on $W$ and $b$ is an
integer such that
\[
 \mu=\int_Wc_1(L)^k>0,\qquad
 \det B=L^{\otimes b}(-D).
\]
Then $b\ge0$, and for every integer $a>b$,
\begin{equation}\label{eq:gap}
 \int_Wc_1(L)^{k-r}c_r(B^\vee\otimes L^{\otimes a})
 \ge a^{r-1}(a-b)\mu>0.
\end{equation}
Consequently, if $B\twoheadrightarrow L^{\otimes a}$ for an integer
$a>0$, then $b\ge a$.
\end{prop}
\begin{proof}
Set $h=c_1(L)$ and $I_i=\int_W h^{k-i}c_i(B)$, $0\le i\le r$.
We regard $\delta=c_1(\calO_W(D))$ as an operational Chow class. Then
$c_1(B)=bh-\delta$ and $\delta\cap[W]=[D]$, with the inclusion
of $D$ understood. Lemma~\ref{lem:schur} and the projection formula give
\[
 0\le I_{i+1}\le\int_W h^{k-i-1}c_1(B)c_i(B)
 =bI_i-\int_Dh^{k-i-1}c_i(B|_D)\le bI_i.
\]
Here $h$ on $D$ means its restriction. For $i=0$ the middle
inequality is equality. For $1\le i<r$ it is the inequality
$c_1c_i-c_{i+1}\ge0$ after taking degree, established in
Lemma~\ref{lem:schur}. All exponents are
nonnegative, and the terms on $D$ are nonnegative on its effective
fundamental cycle. Since $I_0=\mu>0$, the case $i=0$ gives $b\ge0$.

Fix an integer $a>b$, so $a>0$. The splitting principle gives
\[
 \int_W h^{k-r}c_r(B^\vee\otimes L^a)
 =\sum_{i=0}^r(-1)^i a^{r-i}I_i.
\]
For each even $i<r$, the adjacent pair is
$a^{r-i-1}(aI_i-I_{i+1})\ge a^{r-i-1}(a-b)I_i\ge0$.
The first pair is at least $a^{r-1}(a-b)\mu$. In even rank the
remaining term is $+I_r\ge0$. This proves \eqref{eq:gap}.

Now suppose that $B\twoheadrightarrow L^a$ for an integer $a>0$.
Dualizing and twisting gives a trivial line subbundle of
$B^\vee\otimes L^a$ with locally free quotient of rank $r-1$.
Thus $c_r(B^\vee\otimes L^a)=0$. If $a>b$, this contradicts
\eqref{eq:gap}. Hence $b\ge a$.
\end{proof}

\begin{rmk}
The rank and top self-intersection hypotheses are necessary. On $\PP^1$,
$(s^a,t^a)\colon \calO_{\PP^1}^2\twoheadrightarrow\calO_{\PP^1}(a)$
has source determinant degree zero, but its rank exceeds the dimension.
Also, a globally generated line bundle pulled back from a lower-dimensional
variety has zero top self-intersection. The proof requires positive
$I_0$, while the remaining Chern intersections may vanish.
\end{rmk}

\section{Determinant Comparison over a Normal Base}\label{sec:pointed}
We prove the determinant comparison needed to apply
Proposition~\ref{prop:numerical} to a family of lines through a
fixed point. We normalize a graph closure, retaining entire general
fibers, and compare divisors on the projective-line bundle.

\subsection{Extension and Normalization}
\begin{lem}\label{lem:extension}
Let $U$ be a normal variety and let $P$ be a projective scheme. A rational
map $\psi\colon U\dashrightarrow P$ is defined at every codimension-one
point of $U$. Its maximal domain is open, with complement of
codimension at least two. If $\pi\colon W\to U$ is a birational morphism
with $W$ a projective variety, then $\pi$ is an isomorphism over a
neighborhood of each codimension-one point of $U$.
\end{lem}
\begin{proof}
The extension assertion follows from
\cite[\href{https://stacks.math.columbia.edu/tag/0BX7}{Tag 0BX7}]{Stacks},
since $P$ is proper and the codimension-one local rings of $U$
are discrete valuation rings. The extensions glue uniquely on the
maximal domain by separatedness of $P$
\cite[\href{https://stacks.math.columbia.edu/tag/0A1Y}{Tag 0A1Y}]{Stacks}.
Since $W$ is projective and $U$ is separated, $\pi$ is proper, so
the last assertion follows from
\cite[\href{https://stacks.math.columbia.edu/tag/0BFP}{Tag 0BFP}(2)]{Stacks}.
\end{proof}

\begin{lem}\label{lem:graph}
Let $Z$ be a normal projective variety of dimension $k-1$ with $H$ a line bundle on $Z$. Denote by $p\colon U=\PP_Z(\calO_Z\oplus H^{-1})\to Z$ the projective bundle. Suppose a morphism $e_U\colon U\to Y$ maps the
entire section $Z_0$ defined by $\calO_Z$ to $x\in Y_{\sm}$.
There is a normal projective variety $W$ with morphisms $\pi\colon W\to U$ and
$g\colon W\to\widehat Y$ such that $\pi$ is birational,
$\nu g=e_U\pi$, and $\pi$ is an isomorphism near all of $Z_0$
and on $p^{-1}(Z^\circ)$ for some dense open $Z^\circ\subset Z$.
Moreover $g|_{Z_0}$ is constant.
\end{lem}
\begin{proof}
The projective bundle $U$ is normal. The rational lift
$\psi\colon U\dashrightarrow\widehat Y$ is defined on
$e_U^{-1}(Y_{\sm})$, which contains all of $Z_0$. By
Lemma~\ref{lem:extension}, its indeterminacy locus $J$ has
dimension at most $k-2$. Since $p$ is proper, $p(J)$ is closed
of dimension less than $\dim Z$. For $k=1$, $J$ is empty.
Thus $Z^\circ=Z\setminus p(J)$ is dense open and
$p^{-1}(Z^\circ)\subset U\setminus J$.

Let $W$ be the normalization of the reduced graph closure of
$\psi$, with projections $\pi,g$. Finiteness of normalization
\cite[\href{https://stacks.math.columbia.edu/tag/0BXR}{Tag 0BXR}]{Stacks}
makes $W$ projective. Over $U\setminus J$, the graph is already
normal, so $\pi$ is an isomorphism there. This open contains
$Z_0$ and $p^{-1}(Z^\circ)$. The equality $\nu g=e_U\pi$ holds
by agreement on a dense open. Since $\nu$ is an isomorphism
over $x$, the restriction $g|_{Z_0}$ is constant.
\end{proof}

\subsection{Projective-Bundle and Picard Formulas}
We use the notation for the projective bundle and its section from
Lemma~\ref{lem:graph}, allowing $Z$ to be any integral scheme,
and put $\xi=\calO_U(1)$.
The dual tautological map gives
$p^*(\calO_Z\oplus H)\twoheadrightarrow\xi$, so $\xi$ is
globally generated when $H$ is. Write $\tau=c_1(\xi)$ and
$u=p^*c_1(H)$. The tautological line has first Chern class
$-\tau$, and its quotient has first Chern class $\tau-u$.
Whitney's formula therefore gives $\tau^2=\tau u$ as Chern
operations. Since $\xi$ has degree one on each fiber,
$p_*(\tau\cap[U])=[Z]$. Thus, when $Z$ is proper of dimension $k-1$,
\[
 \int_U\tau^k=\int_U\tau u^{k-1}
 =\int_Zc_1(H)^{k-1}.
\]
These are projective-bundle formulas on Chow homology, valid
without smoothness \cite[Chapter~3]{Fulton}.

For a line bundle $M$ on $U$, its degree $b$ on the geometric
fibers is locally constant. Indeed, $p$ is flat and projective
and $M$ is flat over $Z$, so its fiber Hilbert polynomial is locally
constant. Since $Z$ is connected, $b$ is constant.
The bundle $M'=M\otimes\xi^{-b}$ restricts to $\calO_{\PP^1}$
on every geometric fiber. Cohomology and base change yield a
line bundle $P=p_*M'$ and an isomorphism $p^*P\to M'$.
This map is an isomorphism on each fiber, hence everywhere.
Here the fiber dimensions are $h^0=1$ and $h^1=0$. See
\cite[III, \S12]{Hartshorne}. Restricting to $Z_0$ identifies
$P=M|_{Z_0}$, since $\xi|_{Z_0}=\calO_{Z_0}$. Thus
\begin{equation}\label{eq:picard}
 M\simeq\xi^b\otimes p^*(M|_{Z_0}).
\end{equation}
This proof requires neither normality nor smoothness of $Z$.
It applies as well to an integral open subset of the base.

\subsection{A Weil Divisor Avoiding the Section}
The following lemma shows that the pushforward of the determinant
divisor is Cartier.
\begin{lem}\label{lem:vertex-divisor}
Let $Z$ be a normal variety, let $H$ be a line bundle on $Z$,
and set $U=\PP_Z(\calO_Z\oplus H^{-1})$, with section $Z_0$
and $\xi=\calO_U(1)$ as above. If $C$ is an effective Weil divisor on $U$
whose support is disjoint from $Z_0$, then $C$ is an effective
Cartier divisor and
\[
 \calO_U(C)\simeq\xi^q,
\]
where $q$ is the degree of its divisor on the generic fiber.
\end{lem}
\begin{proof}
The variety $U$ is normal because polynomial rings over normal
domains are normal
\cite[\href{https://stacks.math.columbia.edu/tag/030A}{Tag 030A}]{Stacks}.
Work over an affine open $\Spec R\subset Z$ trivializing $H$,
and put $K=\Frac R$. Place $Z_0$ at infinity in $\PP^1_R$.
Write $C=\sum_j m_jC_j$ locally, with prime ideals
$\mathfrak p_j\subset R[z]$.
Let $J_j\subset R[w]$ define $C_j$ in the chart at infinity,
where $w=1/z$ on the overlap. Since $C_j$ avoids $w=0$,
we have $J_j+(w)=R[w]$. Choose $g_j(w)\in J_j$ with $g_j(0)=1$.
On the overlap, $g_j(1/z)\in\mathfrak p_jR[z,z^{-1}]$, so for
sufficiently large $N\ge1$ the monic polynomial $z^Ng_j(1/z)$
belongs to $\mathfrak p_j$. Thus the image $\bar z$ in
$R[z]/\mathfrak p_j$ is integral over $R$, and this quotient
is a finite $R$-module. Since $\dim C_j=\dim R$, each $C_j$
dominates $\Spec R$.

Normality of $R$ is used to ensure that the monic minimal
polynomial $f_j\in K[z]$ of $\bar z$ has coefficients in $R$
\cite[\href{https://stacks.math.columbia.edu/tag/00H7}{Tag 00H7}]{Stacks}.
For $h\in\mathfrak p_j$, division by the monic $f_j$ gives
$h=af_j+r$ with $a,r\in R[z]$ and $\deg r<\deg f_j$.
Evaluating at $\bar z$ gives $r(\bar z)=0$, so minimality of
$f_j$ over $K$ forces $r=0$. Since $f_j\in\mathfrak p_j$,
it follows that $\mathfrak p_j=(f_j)$.
Thus $C$ is defined locally by $\prod_j f_j^{m_j}$ away from
$Z_0$ and is zero near $Z_0$, so it is an effective Cartier divisor.

Since $\calO_U(C)|_{Z_0}\simeq\calO_{Z_0}$ and its degree on the
generic fiber is $q$, the Picard formula \eqref{eq:picard} gives
$\calO_U(C)\simeq\xi^q$.
\end{proof}

\subsection{The Determinant Formula and the Degree Inequality}
\begin{prop}\label{prop:pointed}
Let $Z$ be a normal projective variety of dimension $k-1$, and let
$H$ be a globally generated line bundle on $Z$ with $\int_Zc_1(H)^{k-1}>0$. Put
\[
 p\colon U=\PP_Z(\calO_Z\oplus H^{-1})\longrightarrow Z,
 \qquad \xi=\calO_U(1),
\]
and let $Z_0$ be the section defined by $\calO_Z$. Suppose
$\pi\colon W\to U$ is a projective birational morphism from a normal variety $W$,
and is an isomorphism over a neighborhood of all of $Z_0$.
Identify $Z_0$ with its inverse image in $W$. Let $B$ be a
globally generated vector bundle of rank $1\le r\le k$ on $W$, with
\[
 (\det B)|_{Z_0}\simeq\calO_{Z_0},\qquad
 B\twoheadrightarrow\pi^*\xi^a,\quad a>0.
\]
The morphism $\pi$ is an isomorphism over a general fiber of the
projective-line bundle. If $b$ is the degree of $\det B$ on its
inverse image, then $b\ge a$.
\end{prop}
\begin{proof}
The normal variety $U$ is regular at its codimension-one points.
Lemma~\ref{lem:extension} shows that $\pi$ is an isomorphism over
a neighborhood of every such point. The complement $J$ of its
isomorphism locus therefore has dimension at most $k-2$.
Since $p$ is proper, $p(J)$ is a proper closed subset of $Z$.
Thus each fiber over $Z\setminus p(J)$ is isomorphic to its full
inverse image. The degree of the determinant on these fibers is
constant by the projective-bundle Picard formula \eqref{eq:picard}.
For $k=1$, $J$ is empty.

The projective-bundle formula gives, with $\tau=c_1(\xi)$,
\begin{equation}\label{eq:ruled-identities}
 \xi\text{ globally generated},\qquad
 \int_U\tau^k=\int_Zc_1(H)^{k-1}>0,\qquad
 \Pic(U)=p^*\Pic(Z)\oplus\ZZ[\xi].
\end{equation}
These identities do not require smoothness.
Choose a section of $\det B$ whose zero divisor $A$ misses $Z_0$
and contains no $\pi$-exceptional prime divisor. Global generation
excludes the finitely many prohibited linear conditions. The
exceptional primes are codimension-one components of a fixed proper
closed subset, hence finite in number. Each condition of vanishing
identically on one of them defines a proper linear subspace.
The entire section is avoided because a nonzero restriction to
its trivial determinant is constant on the projective variety $Z_0$.

The effective Weil divisor $C=\pi_*A$ misses $Z_0$.
Lemma~\ref{lem:vertex-divisor} shows that $C$ is Cartier and
$\calO_U(C)=\xi^q$, where $q$ is its degree on the generic fiber.
On a general fiber, $C$ and $A$ agree under the isomorphism induced by
$\pi$, so $q=b$.

Set $D=\pi^*C-A$. At nonexceptional prime divisors its coefficients
are zero. At exceptional prime divisors the coefficient of $A$
is zero and that of $\pi^*C$ is nonnegative. Since $W$ is normal,
the nonnegative Weil coefficients of the Cartier divisor $D$
imply that $D$ is effective
\cite[\href{https://stacks.math.columbia.edu/tag/031T}{Tag 031T}]{Stacks}.
Hence
\[
 \det B=\pi^*\xi^b(-D),\qquad D\text{ an effective Cartier divisor}.
\]
Apply Proposition~\ref{prop:numerical} to $L=\pi^*\xi$.
The projection formula and \eqref{eq:ruled-identities} give its
positive top self-intersection. For $r=1$, the quotient itself is an
isomorphism, so $b=a$.
\end{proof}

\begin{rmk}\label{rem:vertex}
Take $Z=\PP^1$, $H=\calO_Z(1)$ and
$U=\PP_Z(\calO_Z\oplus H^{-1})$. The globally generated rank-two
bundle $B=p^*(\calO_Z\oplus H)$ has its tautological quotient
$B\twoheadrightarrow\xi$. The degree of its determinant on a fiber is
$b=0$, although the quotient exponent is $a=1$ and $\dim U=2$.
Here $(\det B)|_{Z_0}=H$ is nontrivial, so the determinant
triviality assumption in Proposition~\ref{prop:pointed} is necessary.
\end{rmk}

\section{Dimension Bounds, Freeness, and Irreducibility}\label{sec:proof}
We apply the two determinant estimates to a family through a smooth
point of $Y$, then deduce the dimension bounds, freeness, and
irreducibility for $F_1(X)$.

\begin{proof}[Proof of Theorem~\ref{thm:swept}]
If $m\le n-3$, the assertion is immediate. Assume $m\ge n-2$
and set $k=t+1=m+2-y\ge r$.
The map $e|_\Omega\colon \Omega\to Y_{\sm}$ in Proposition~\ref{prop:line} is
smooth of relative dimension $t$. Choose a closed point $x$ in
its nonempty open image. Take a component of $\Omega_x$ and let
$T$ be its reduced closure in the proper fiber $e^{-1}(x)$.
Then $T$ is a projective variety of dimension $t$, and its general
point parametrizes a member of $S^\circ$.

If $L_x\subset V$ represents $x$, the scheme of lines through
$x$ is $\PP(V/L_x)\subset\Gr(2,n+1)$, and $e^{-1}(x)$ is its fiber
product with $S$. Thus $T$ is a closed subvariety of this direction
space. Let $\rho\colon Z=T^\nu\to T$ be its normalization, a finite
birational morphism, and put $H=\rho^*\calO_T(1)$.
The variety $Z$ is normal and projective, but need not be smooth.
Then
\[
 \dim Z=t=k-1,\qquad H\text{ globally generated},\qquad
 \int_Zc_1(H)^{k-1}=\deg T>0.
\]
For $t=0$ these assertions mean $Z=T$ is a point and its degree is one.

A splitting $V=L_x\oplus V'$ identifies the tautological line
bundle with $H^{-1}\subset V'\otimes\calO_Z$. Hence
$U=\PP_Z(\calO_Z\oplus H^{-1})$ has a morphism $e_U\colon U\to Y$
with $e_U^*\calO_Y(1)=\xi$. Each fiber maps isomorphically
to the line it parametrizes, and $Z_0$ maps to $x$.

Normalize the reduced graph closure of the rational lift
$U\dashrightarrow\widehat Y$ to obtain a normal projective variety
$W$ with morphisms
\[
 \pi\colon W\longrightarrow U,\qquad g\colon W\longrightarrow\widehat Y,
 \qquad \nu g=e_U\pi.
\]
By Lemma~\ref{lem:graph}, $\pi$ is an isomorphism over a
neighborhood of the whole section $Z_0$ and over entire general
fibers of the projective-line bundle. Since $x\in Y_{\sm}$,
$g|_{Z_0}$ is constant.

Now $B=g^*\calB_Y$ is globally generated of rank $r\le k=\dim W$,
is trivial on $Z_0$, and \eqref{eq:gradient} pulls back to
\[
 B\twoheadrightarrow g^*\nu^*\calO_Y(d-1)
 =(e_U\pi)^*\calO_Y(d-1)=\pi^*\xi^{d-1}.
\]
Proposition~\ref{prop:pointed} gives $b\ge d-1$ on the inverse
image of a general fiber. Choose this fiber also over the open
set of $Z$ parametrizing $S^\circ$. Since $\pi$ is an isomorphism
over the whole fiber, the restriction of $g$ defines a lift of $\ell$.
It agrees with $\lambda_\ell$ over $\ell\cap Y_{\sm}$, hence on
the whole line by uniqueness. The two constructions therefore give
the same determinant degree, and Proposition~\ref{prop:line} yields
\begin{equation}\label{eq:sandwich}
 d-1\le b=\deg\lambda_\ell^*\calB_Y\le2y-m-2.
\end{equation}
Both choices impose dense open conditions on the integral base $Z$,
so the same fiber satisfies them.
Hence $m\le2y-d-1$, proving \eqref{eq:swept}.
Since $y\le n-1$, we also obtain
\begin{equation}\label{eq:uniform}
 m\le\max\{2n-d-3,\,n-3\}.
\end{equation}
\end{proof}

\begin{rmk}\label{rem:smooth-swept}
If $Y$ is smooth, the lower bound can be proved on the original
$U_T=\PP_T(\calO_T\oplus\calO_T(-1))$, even when $T$ is nonnormal.
The bundle $B=e_T^*N_{Y/\PP^n}(-1)$ is globally generated,
admits the gradient map $B\twoheadrightarrow\xi^{d-1}$ and is
constant on the section $T_0$ contracted to the fixed point.
Formula~\eqref{eq:picard} therefore gives the exact identity
$\det B=\xi^b$. The projective-bundle formula gives
$\int_{U_T}c_1(\xi)^k=\deg T>0$, so
Proposition~\ref{prop:numerical} applies with $D=0$.
On a general line, $M_\ell=N_{\ell/Y}$ and the upper bound is
the usual normal-bundle deformation calculation.
The same determinant calculation applies whenever the rational lift
extends to $U_T$.
For singular $Y$, the normalizations above extend this lift and
yield an effective divisor in the determinant formula.
\end{rmk}

\begin{proof}[Proof of Theorem~\ref{thm:main}]
By Barth--Van de Ven \cite[Theorem~8(i)]{BVV}, every degree-$d$
hypersurface in $\PP^n$ contains a line if $d\le2n-3$.
For $d\le n$, each component has dimension at least
$\delta=2n-d-3$ because it is cut out by $d+1$ local equations in
$\Gr(2,n+1)$. Give its support the reduced structure and apply
\eqref{eq:uniform}. Since $\delta\ge n-3$, its dimension is at
most $\delta$. Equality follows.

For $d\ge n+1$, \eqref{eq:uniform} gives $m\le n-3$
for every component. The classical Fermat example recalled in
\S\ref{subsec:fermat} shows that this bound is sharp.
\end{proof}

Theorem~\ref{thm:swept} also proves Debarre's covering conjecture
\cite[Conjecture~2]{Debarre03}.
\begin{cor}\label{cor:debarre-covering}
Let $X\subset\PP^n_\CC$ be a smooth hypersurface of degree
$d\ge2$, with $n\ge3$. If $d\ge n-2$ and
$S\subset F_1(X)$ is a closed subvariety of dimension at least
$n-2$, then the lines parametrized by $S$ cover $X$.
\end{cor}
\begin{proof}
Let $Y$ be the reduced union of the lines parametrized by $S$,
and put $m=\dim S$ and $y=\dim Y$. If $Y\ne X$, then $y\le n-2$.
Theorem~\ref{thm:swept} and $d\ge n-2$ give
\[
 m\le\max\{2y-d-1,n-3\}\le n-3,
\]
contrary to $m\ge n-2$.
\end{proof}

\subsection{Freeness and Irreducibility}\label{subsec:free-irreducible}
For a line $\ell\subset X$, smoothness of $X$ gives the normal sequence
\begin{equation}\label{eq:normal}
 0\longrightarrow N_{\ell/X}\longrightarrow
 \calO_\ell(1)^{n-1}\longrightarrow\calO_\ell(d)\longrightarrow0.
\end{equation}
Every summand of $N_{\ell/X}$ has degree at most one, so
$\ell$ is free if and only if $H^1(N_{\ell/X}(-1))=0$.

\begin{proof}[Proof of Corollary~\ref{cor:free-irreducible}]
Assume $d\le n-1$. For each reduced component $S\subset F_1(X)$,
Theorem~\ref{thm:main} gives $m=2n-d-3\ge n-2$, and
Theorem~\ref{thm:swept} gives $m\le2y-d-1$. Hence $y=n-1$
and every component sweeps out $X$.

For any integral covering family on $X$, generic smoothness of
evaluation shows that sections arising from infinitesimal
deformations generically generate $N_{\ell/X}$ on a general line.
Its splitting on $\PP^1$ therefore has no negative summand, so it is globally generated
\cite[Theorem~2.1]{Grothendieck57}. This proves \textup{(i)}
by the standard deformation argument for covering families
\cite[Corollary~II.3.10.1]{Kollar}, also recalled in
\cite[Proposition~2.2]{BR}.

The nonfree locus is closed by upper semicontinuity of
$h^1(N_{\ell/X}(-1))$ in the universal normal sequence
\eqref{eq:normal}. No component of this locus sweeps out $X$,
by the preceding argument applied to its reduced structure.
Its finitely many evaluation images are proper closed subsets,
since the universal line over each component is projective.
Removing their union gives a dense open $X^\circ\subset X$
through which every line is free.

Now assume $d\le n-2$ and choose $x\in X^\circ$.
At every line $\ell$ through $x$, freeness gives
$H^1(N_{\ell/X})=H^1(N_{\ell/X}(-x))=0$.
The first vanishing and \eqref{eq:normal} show that $F_1(X)$ is
smooth at $[\ell]$. Global sections of the normal bundle, together
with the tangent direction of the line, surject onto $T_xX$,
so the evaluation map to $X$ is smooth at $(\ell,x)$.
Since $N_{\ell/X}$ has rank $n-2$ and degree $n-d-1$, the tangent
space to the fiber, $H^0(N_{\ell/X}(-x))$, has dimension
$n-d-1$. Thus the nonempty fiber $F_1(X,x)$ is smooth of
positive dimension $n-d-1$.

In coordinates $x=[1:0:\cdots:0]$, write
$F=\sum_{i=1}^d z_0^{d-i}f_i(z_1,\ldots,z_n)$, with $\deg f_i=i$.
The equations for the lines through this point are
\[
 F_1(X,x)=V(f_1,\ldots,f_d)\subset\PP^{n-1}.
\]
The dimension just computed gives codimension $d$, so this is a
positive-dimensional complete intersection. Its Koszul resolution
shows that $H^0(\calO_{F_1(X,x)})=\CC$. Indeed, writing $N=n-1$, the
negative line-bundle terms have no cohomology in degrees
$0,\ldots,N-1$, and their degree-$N$ cohomology contributes only
to total degree at least $N-d\ge1$. Hence the fiber is connected
and, being smooth, irreducible. This is the connectedness argument
used in \cite[Corollary~2.7]{BR}.

Every component of $F_1(X)$ sweeps out $X$, so properness of
evaluation makes its contribution to $F_1(X,x)$ nonempty and
closed. Distinct components contribute disjoint subsets, since
$F_1(X)$ is smooth at every line through $x$. Connectedness of
the fiber therefore forces a single component, proving \textup{(ii)}.
\end{proof}
A smooth quadric surface has two rulings, so the irreducibility
range is sharp as a uniform statement.

\subsection{Classical Sharpness Examples}
\label{subsec:fermat}
We recall the standard Fermat example, recorded in
\cite[p.~1]{Debarre03} and described explicitly in \cite[\S2]{B06}.
On $X_d=V(z_0^d+\cdots+z_n^d)$, take
$p=[1:\zeta:0:\cdots:0]$, with $\zeta^d=-1$.
The lines joining $p$ to
\[
 Q_d=V(z_2^d+\cdots+z_n^d)\subset\PP^{n-2}
\]
lie on $X_d$. For $n\ge4$, they form a family isomorphic to $Q_d$
of dimension $n-3$, sweeping a cone of dimension $n-2$.
For $n=3$, any point of $Q_d$ supplies a line. This is the classical
example attaining the bound $n-3$ for $d\ge n$. See also
\cite[\S1]{B06} and \cite[\S3.1]{BR}.

The classical incidence count also shows that a general hypersurface
contains no line when $d\ge2n-2$. Indeed, the universal incidence
variety of \cite[Theorem~3.3]{AK} has dimension
$\dim H^0(\calO(d))+2n-d-3$, smaller than the space of forms,
and its projection is proper.

We record how the Fermat example relates to our degree inequality.
In our notation it has $m=n-3$, $y=n-2$, $r=2$, and $t=0$.
At $d=n$, one has $m=n-3>n-5=2y-d-1$, so the term $n-3$
in \eqref{eq:swept} cannot be omitted. The subspace defining the
Gauss map is constant along a ruling, giving $\lambda_\ell^*\calB_Y=\calO_\ell^2$
and $b_\ell=0$, while the gradient map is
\[
 (s^{d-1},t^{d-1})\colon \calO_\ell^2\twoheadrightarrow\calO_\ell(d-1).
\]
Over a general smooth point of the swept variety, the construction
gives a projective bundle of dimension $k=1<r=2$, outside the
hypotheses of Proposition~\ref{prop:pointed}.
For $n\ge4$, the cone vertex is singular on the swept variety,
so it cannot replace the smooth point chosen in that proposition's
application.

\section{Comparison with Earlier Approaches}\label{sec:comparison}
We compare the degree of the tautological quotient with
Beheshti--Riedl's normal-bundle invariant and apply the determinant
inequality to families of nonfree lines.

\subsection{Normal-Bundle Invariants}
For a general line $\ell$ of a family, put
\[
 \delta_{\mathrm{BR}}=n-1-h^0(N_{\ell/X}(-1)).
\]
Beheshti--Riedl expand the defining form at a point to obtain
equations of degrees $1,\ldots,d$ for the lines through it, then
apply a tangency-locus dimension bound
\cite[Lemmas~2.1 and 2.3]{BR}. Their Theorem~2.4 gives
\begin{equation}\label{eq:BR-dichotomy}
 \begin{split}
 &\delta_{\mathrm{BR}}=d,\quad N_{\ell/X}\text{ globally generated},
 \qquad\text{or}\\
 &t\le\delta_{\mathrm{BR}}-1,\qquad
 m\le y+\delta_{\mathrm{BR}}-2.
 \end{split}
\end{equation}
For a hypothetical component of dimension greater than expected, their proof of
\cite[Corollary~2.5]{BR} uses \cite[Theorem~3.2(b)]{B14} to obtain
$y\le n-3$. The case $\delta_{\mathrm{BR}}=d-1$ is unobstructed,
so the remaining case satisfies
\[
 m\le(n-3)+(d-2)-2=n+d-7\le2n-d-3
 \quad\Longleftrightarrow\quad n\ge2d-4.
\]
They prove irreducibility for $n\ge2d-1$ and $n\ge4$
\cite[Corollary~2.8]{BR}. We compare their invariant with the
degree of the tautological quotient through the following exact sequence.

\begin{prop}\label{prop:bridge}
For a general line as in Proposition~\ref{prop:line}, write
$b=\deg\calB_\ell$ and
$K_\ell=\ker(\calB_\ell\twoheadrightarrow\calO_\ell(d-1))$.
Then $K_\ell$ is locally free of rank $r-1$ and degree $b-d+1$,
and
\begin{equation}\label{eq:bridge-sequence}
 0\longrightarrow M_\ell(-1)\longrightarrow
 N_{\ell/X}(-1)\longrightarrow K_\ell\longrightarrow0.
\end{equation}
Moreover
\begin{equation}\label{eq:bridge}
 \delta_{\mathrm{BR}}=r+b-h^0(K_\ell)=d-h^1(K_\ell).
\end{equation}
In particular, $\delta_{\mathrm{BR}}=d$ if and only if $\ell$ is free.
\end{prop}
\begin{proof}
The map $\calO_\ell^{n-1}\to\calO_\ell(d-1)$ in the twist of
\eqref{eq:normal} factors through $\calB_\ell$ by \eqref{eq:gradient}.
Taking kernels gives \eqref{eq:bridge-sequence}, so
$M_\ell\subset N_{\ell/X}$ is a subbundle with quotient $K_\ell(1)$.
Since $M_\ell(-1)=\calO_\ell^{y-1-b}\oplus\calO_\ell(-1)^b$,
its first cohomology vanishes and its space of sections has dimension
$y-1-b$. The cohomology sequence therefore gives
\[
 h^0(N_{\ell/X}(-1))=y-1-b+h^0(K_\ell).
\]
This proves the first equality in \eqref{eq:bridge}. Riemann--Roch
on $\PP^1$ gives $\chi(K_\ell)=r+b-d$, proving the second.
The same cohomology sequence identifies
$H^1(K_\ell)=H^1(N_{\ell/X}(-1))$, which vanishes exactly for a
free line, by the splitting criterion following \eqref{eq:normal}.
\end{proof}

The degree $b$ depends on the limiting tangent spaces of $Y$,
whereas $\delta_{\mathrm{BR}}$ depends on $\ell\subset X$.
For example, $Y=X$ gives $b=d-1$ and $\delta_{\mathrm{BR}}=d$.
Our determinant inequality gives
\[
 k-r=m+2-n\ge0
 \quad\Longrightarrow\quad
 d-1\le b\le2y-m-2
 \quad\Longrightarrow\quad t\le y-d.
\]
The condition $m\ge n-2$ holds for every component when $d<n$,
and for any component exceeding the asserted bound when $d\ge n$.
Writing $c=\codim_XY$, we obtain
\[
 m\le\max\{(2n-d-3)-2c,\,n-3\}.
\]
For $d\le n-1$, this forces every component of $F_1(X)$ to dominate
$X$, as used in the proof of Corollary~\ref{cor:free-irreducible}.
For a nonfree family with $m\ge n-2$, the two bounds are
$m\le y+\delta_{\mathrm{BR}}-2$ and $m\le2y-d-1$.
The second is strictly smaller precisely when
$y<d+\delta_{\mathrm{BR}}-1$.

\subsection{Families of Nonfree Lines}
The standard deformation argument for covering families
\cite[Corollary~II.3.10.1]{Kollar}, recalled in
\cite[Proposition~2.2]{BR}, shows that a family with nonfree general
member sweeps a proper subvariety of $X$. Under $m\ge n-2$, we obtain
an additional codimension.

\begin{prop}\label{prop:nonfree-codimension}
Let $S\subset F_1(X)$ be a projective subvariety with
$m\ge n-2$. Suppose the normal bundle of a general member has $u\ge1$
negative summands in its splitting on $\PP^1$. Then
\[
 \codim_XY\ge u+1\ge2.
\]
In particular a family of this dimension sweeping either $X$ or a
divisor in $X$ has a free general member.
\end{prop}
\begin{proof}
Shrink the open subset $S^\circ$ so that the splitting type of
$N_{\ell/X}$ is the general one. Repeat the construction in the
proof of Theorem~\ref{thm:swept} over this smaller open subset.
The chosen line has a normal bundle with exactly $u$ negative
summands and satisfies $b\ge d-1$ by \eqref{eq:sandwich}. Thus
$\deg K_\ell\ge0$. Write
$N_{\ell/X}=N_+\oplus N_-$, where $N_+$ is the sum of the nonnegative
summands and $N_-$ has rank $u$. The globally generated bundle $M_\ell$
has no nonzero map to $N_-$, so its inclusion factors through $N_+$.
Twisting \eqref{eq:bridge-sequence} gives
\[
 K_\ell(1)=(N_+/M_\ell)\oplus N_-.
\]
The quotient $R=N_+/M_\ell$ is locally free, since it is a direct
summand of the locally free left side. It is also globally generated.
If $R$ had rank zero, then $K_\ell=N_-(-1)$ would have negative degree,
a contradiction. Hence $r-1-u=\rank R\ge1$, and
$\codim_XY=r-1\ge u+1$.
\end{proof}

\begin{cor}\label{cor:nonfree}
Let $\NF(X)$ be the reduced closed locus of nonfree
lines on $X$. If it is nonempty, then
\begin{equation}\label{eq:NF-independent}
 \dim\NF(X)\le\max\{n-3,\,2n-d-7\}.
\end{equation}
If $2\le d\le n-1$, this gives
\begin{equation}\label{eq:NF-codimension}
 \codim_{F_1(X)}\NF(X)\ge\min\{4,n-d\}.
\end{equation}
Combining the same argument with \cite[Theorem~2.4]{BR} gives
\begin{equation}\label{eq:NF-hybrid}
 \dim\NF(X)
 \le\max\{n-3,\,\min(2n-d-7,\,n+d-6)\}.
\end{equation}
In particular, $\dim\NF(X)\le n+d-6$ when $d\ge3$.
\end{cor}
\begin{proof}
The nonfree locus is closed by the semicontinuity argument in
\S\ref{subsec:free-irreducible}.
Give each of its irreducible components $S$ the reduced structure.
If $m<n-2$, then $m\le n-3$. Otherwise
Proposition~\ref{prop:nonfree-codimension} gives $y\le n-3$, and
Theorem~\ref{thm:swept} gives $m\le2n-d-7$. This proves
\eqref{eq:NF-independent}. In the Fano range of the second assertion,
Theorem~\ref{thm:main} gives pure dimension $2n-d-3$. Subtracting
\eqref{eq:NF-independent} yields \eqref{eq:NF-codimension}.

For the last assertion, a nonfree line has
$\delta_{\mathrm{BR}}\le d-1$ by \eqref{eq:bridge}. If $m\ge n-2$,
\eqref{eq:BR-dichotomy} and $y\le n-3$ give
$m\le y+\delta_{\mathrm{BR}}-2\le n+d-6$. Combining this with
the previous estimate and the case $m\le n-3$ proves
\eqref{eq:NF-hybrid}. For $d\ge3$ one has $n-3\le n+d-6$.
\end{proof}

Proposition~2.2 and Theorem~2.4 of \cite{BR} give
$\dim\NF(X)\le n+d-5$.
This bound also appears in \cite[Proposition~5.2]{BLLRST},
under the standing assumptions $d\le n$ and $n\ge4$.
For $d\ge3$, \eqref{eq:NF-hybrid} improves it by at least one
dimension. The independent bound \eqref{eq:NF-independent}
is not uniformly smaller than $n+d-5$.

The classical Fermat family recalled in \S\ref{subsec:fermat}, with $d\ge3$,
has $m=n-3$ and $\codim_XY=1$. The entries of the map in its
normal sequence span the space generated by $s^{d-1}$
and $t^{d-1}$. The kernel of
$\calO_\ell(1)^2\twoheadrightarrow\calO_\ell(d)$ is
$\calO_\ell(2-d)$, so
\[
 N_{\ell/X}=\calO_\ell(1)^{n-3}\oplus\calO_\ell(2-d).
\]
Thus its members are nonfree. Both the dimension threshold and the
$n-3$ term in \eqref{eq:NF-independent} are necessary.

The dimension and nonfree-locus estimates also give the following
scheme-theoretic consequences.
\begin{cor}\label{cor:scheme-structure}
Let $X\subset\PP^n_\CC$ be a smooth hypersurface of degree
$2\le d\le n$, with $n\ge3$.
\begin{enumerate}[(i)]
\item The scheme $F_1(X)$ is a nonempty local complete intersection
of pure dimension $2n-d-3$. In particular, it is Cohen--Macaulay.
\item If $d\le n-1$, then $F_1(X)$ is generically smooth and reduced.
\item If $d\le n-2$, then $F_1(X)$ is integral and normal.
\end{enumerate}
\end{cor}
\begin{proof}
The scheme $F_1(X)$ is the zero scheme of a section of the
rank-$(d+1)$ bundle $\Sym^d\calS^\vee$ on $\Gr(2,n+1)$.
By Theorem~\ref{thm:main}, it is nonempty and has pure
codimension $d+1$. Its local defining equations therefore form
a regular sequence in the regular ambient local rings
\cite[\href{https://stacks.math.columbia.edu/tag/02JN}{Tag 02JN}]{Stacks}.
Hence it is a local complete intersection and Cohen--Macaulay,
proving \textup{(i)}.

If $d\le n-1$, Corollary~\ref{cor:free-irreducible} gives a free
general line in every component. For such a line,
$H^1(N_{\ell/X})=0$, so \eqref{eq:normal} shows that $F_1(X)$
is smooth at $[\ell]$. Thus it is generically smooth and
satisfies $(R_0)$. Since it is Cohen--Macaulay, it also
satisfies $(S_1)$ and is therefore reduced
\cite[\href{https://stacks.math.columbia.edu/tag/031R}{Tag 031R}]{Stacks}.
This proves \textup{(ii)}.

If $d\le n-2$, Corollary~\ref{cor:free-irreducible} also gives
irreducibility, hence integrality. Every free line is a smooth
point of $F_1(X)$, and \eqref{eq:NF-codimension} shows that the
nonfree locus, if nonempty, has codimension at least two.
Thus $F_1(X)$ satisfies $(R_1)$ and, being Cohen--Macaulay,
$(S_2)$. Serre's criterion gives normality
\cite[\href{https://stacks.math.columbia.edu/tag/031S}{Tag 031S}]{Stacks},
proving \textup{(iii)}.
\end{proof}

\subsection{Tangency and Equations for Lines through a Point}
\cite[Lemma~2.1]{BR} bounds the tangency locus of a degree-$d$
hypersurface $V(h)$ with $\dim\Sing V(h)=s$ to $q$
lower-degree equations by $q+s$. The bound remains sharp even
where the lower-degree equations have independent differentials,
and even when all their degrees are prescribed. In $\PP^{2q-1}$ take
\[
 h=\sum_{i=1}^q(x_i^d+x_i y_i^{d-1}),\qquad
 h_i=x_i y_i^{e_i-1},\quad 1\le e_i<d.
\]
The hypersurface $V(h)$ is smooth because the two partial derivatives in
each pair $(x_i,y_i)$ vanish simultaneously only at $(0,0)$.
On the open subset $y_1\cdots y_q\ne0$ of
$L=V(x_1,\ldots,x_q)\simeq\PP^{q-1}$, the $dh_i$ are independent
multiples of $dx_i$, and $dh$ lies in their span. The tangency locus
therefore has dimension at least $q-1$, attaining $q+s$ for $s=-1$.
Adding $s+1$ unused coordinates gives a cone with singular locus
$\PP^s$ and a tangency locus of dimension $q+s$. Thus the original
hypotheses allow no uniformly smaller bound.

We recover the rank calculation of \cite[Lemma~2.3]{BR} from
the normal sequence. In coordinates
$\ell=V(z_2,\ldots,z_n)$, write
\[
 F=\sum_{j=2}^n z_j g_j(s,t)\pmod{\calI_\ell^2},
 \qquad \deg g_j=d-1.
\]
The normal sequence gives
$\delta_{\mathrm{BR}}=\dim\langle g_2,\ldots,g_n\rangle$.
The linear parts of the equations for lines through a general point
of $\ell$ are given by successive derivatives of these binary forms.
More explicitly,
put $q_j(t)=g_j(1,t)$ and take $p=(1,a)$. Then the $i$th row has entries
$q_j^{(i-1)}(a)/(i-1)!$. Their first
$\delta_{\mathrm{BR}}$ rows have that rank in characteristic zero because
a basis with distinct degrees on an affine chart has a nonzero
Wronskian, whose leading coefficient is a nonzero Vandermonde
factor in those degrees. All further rows lie in the same row span.
Hence the first $\delta_{\mathrm{BR}}$ equations have independent
linear parts, and every later linear part lies in their span.

For $n\ge d+2$, \cite[Corollary~2.7]{BR} gives a smooth irreducible
family of lines through a general point. Our irreducibility proof
uses the same connectedness argument together with dominance of
every component of $F_1(X)$.

\subsection{The Chern-Class Approach}
Landsberg--Tommasi \cite{LT} construct vector bundles using normal
forms for spaces of matrices of constant rank two arising from
tangent spaces to the Fano scheme. Their reduction detects singular
points through a product of top Chern classes. They also treat a
special case involving one bundle.

We use the tautological quotient bundle on the Nash modification.
The surjective gradient map gives $c_r(B^\vee\otimes L^{d-1})=0$,
while the determinant comparison and Lemma~\ref{lem:schur} give
\[
 \det B=L^b(-D),\qquad D\ge0,\qquad
 I_{i+1}\le bI_i,\qquad I_0>0.
\]
Pairing consecutive terms in the expansion of the top Chern class
forces $b\ge d-1$. The same bundle determines the degree on a complete
line and the determinant on the normalized graph closure.
Establishing nonvanishing of the product of top Chern classes,
and treating the singular hypersurfaces considered in \cite{LT},
require further arguments.

Beheshti's arguments \cite{B06,B14} use the geometry of families
sweeping divisors, with Hodge theory and monodromy in the later work.
We obtain dominance directly from Theorem~\ref{thm:swept}.
Related constructions include the logarithmic tangent bundles for
highly tangent lines in \cite{PRST} and the use of collective
strength and the Fano transfer principle for Fano schemes of
higher planes and complete intersections in \cite[Theorem~1.3]{ER}.
Here we consider contained lines on smooth complex hypersurfaces.

\section*{Acknowledgements}
The author is grateful to AI models, especially ChatGPT for valuable assistance with the exploration of this topic and preparation of the paper. The author takes full responsibility for its content and correctness.

\bibliography{research_next}

\newcommand{\etalchar}[1]{$^{#1}$}
\providecommand{\bysame}{\leavevmode\hbox to3em{\hrulefill}\thinspace}
\providecommand{\MR}{\relax\ifhmode\unskip\space\fi MR }
\providecommand{\MRhref}[2]{%
  \href{http://www.ams.org/mathscinet-getitem?mr=#1}{#2}
}
\providecommand{\href}[2]{#2}
\begin{thebibliography}{BLL{\etalchar{+}}24}

\bibitem[AK77]{AK}
A.~B. Altman and S.~L. Kleiman, \emph{Foundations of the theory of {Fano}
  schemes}, Compos. Math. \textbf{34} (1977), no.~1, 3--47,
  \href{https://www.numdam.org/item/CM_1977__34_1_3_0/}{Journal archive}.

\bibitem[Beh05]{B05}
R.~Beheshti, \emph{Linear subvarieties of hypersurfaces}, Int. Math. Res. Not.
  (2005), no.~49, 3055--3063,
  \href{https://doi.org/10.1155/IMRN.2005.3055}{doi:10.1155/IMRN.2005.3055}.

\bibitem[Beh06]{B06}
\bysame, \emph{Lines on projective hypersurfaces}, J. Reine Angew. Math.
  \textbf{592} (2006), 1--21,
  \href{https://doi.org/10.1515/CRELLE.2006.020}{Publisher DOI}.

\bibitem[Beh14]{B14}
\bysame, \emph{Hypersurfaces with too many rational curves}, Math. Ann.
  \textbf{360} (2014), no.~3--4, 753--768,
  \href{https://doi.org/10.1007/s00208-014-1024-8}{doi:10.1007/s00208-014-1024-8}.
  \href{https://www.math.wustl.edu/~beheshti/hl.pdf}{Author version}.

\bibitem[BLL{\etalchar{+}}24]{BLLRST}
R.~Beheshti, B.~Lehmann, C.~Lian, E.~Riedl, J.~Starr, and S.~Tanimoto, \emph{On
  the asymptotic enumerativity property for {Fano} manifolds}, Forum Math.
  Sigma \textbf{12} (2024), e112,
  \href{https://doi.org/10.1017/fms.2024.87}{doi:10.1017/fms.2024.87}.

\bibitem[BR21]{BR}
R.~Beheshti and E.~Riedl, \emph{Linear subspaces of hypersurfaces}, Duke Math.
  J. \textbf{170} (2021), no.~10, 2263--2288,
  \href{https://doi.org/10.1215/00127094-2021-0035}{doi:10.1215/00127094-2021-0035}.

\bibitem[BvdV78]{BVV}
W.~Barth and A.~van~de Ven, \emph{{Fano}-varieties of lines on hypersurfaces},
  Arch. Math. (Basel) \textbf{31} (1978), 96--104,
  \href{https://doi.org/10.1007/BF01226420}{doi:10.1007/BF01226420}.

\bibitem[Deb03]{Debarre03}
O.~Debarre, \emph{Lines on smooth hypersurfaces}, Unpublished notes, 2003,
  November 2003, with a 2006 addendum.
  \href{https://perso.imj-prg.fr/olivier-debarre/wp-content/uploads/debarre-pub/Lines_hypersurfaces.pdf}{Author
  manuscript}.

\bibitem[ER24]{ER}
D.~Erman and E.~Riedl, \emph{A transfer principle for unirationality},
  arXiv:2410.20051, 2024, \href{https://arxiv.org/abs/2410.20051}{Author
  manuscript}.

\bibitem[FL83]{FL}
W.~Fulton and R.~Lazarsfeld, \emph{Positive polynomials for ample vector
  bundles}, Ann. of Math. (2) \textbf{118} (1983), no.~1, 35--60,
  \href{https://doi.org/10.2307/2006953}{doi:10.2307/2006953}.

\bibitem[Ful98]{Fulton}
W.~Fulton, \emph{{Intersection Theory}}, second ed., Ergebnisse der Mathematik
  und ihrer Grenzgebiete, 3.~Folge, vol.~2, Springer-Verlag, Berlin, 1998,
  \href{https://doi.org/10.1007/978-1-4612-1700-8}{doi:10.1007/978-1-4612-1700-8}.

\bibitem[Gro57]{Grothendieck57}
A.~Grothendieck, \emph{Sur la classification des fibr\'es holomorphes sur la
  sph\`ere de {Riemann}}, Amer. J. Math. \textbf{79} (1957), no.~1, 121--138,
  \href{https://doi.org/10.2307/2372388}{doi:10.2307/2372388}.

\bibitem[Gro61]{FGAHilbert}
\bysame, \emph{Techniques de construction et th\'eor\`emes d'existence en
  g\'eom\'etrie alg\'ebrique iv : les sch\'emas de {Hilbert}}, S\'eminaire
  Bourbaki : ann\'ees 1960/61, expos\'es 205--222, S\'eminaire Bourbaki, no.~6,
  Soci\'et\'e Math\'ematique de France, 1961, Expos\'e no.~221.
  \href{https://www.numdam.org/item/SB_1960-1961__6__249_0/}{Original article},
  pp.~249--276.

\bibitem[Har77]{Hartshorne}
R.~Hartshorne, \emph{{Algebraic Geometry}}, Graduate Texts in Mathematics,
  vol.~52, Springer-Verlag, New York, 1977,
  \href{https://doi.org/10.1007/978-1-4757-3849-0}{doi:10.1007/978-1-4757-3849-0}.

\bibitem[Kle74]{Kleiman}
S.~L. Kleiman, \emph{The transversality of a general translate}, Compos. Math.
  \textbf{28} (1974), no.~3, 287--297,
  \href{https://www.numdam.org/item/CM_1974__28_3_287_0/}{Journal archive}.

\bibitem[Kol96]{Kollar}
J.~Koll{\'a}r, \emph{Rational curves on algebraic varieties}, Ergebnisse der
  Mathematik und ihrer Grenzgebiete, 3.~Folge, vol.~32, Springer-Verlag,
  Berlin, 1996,
  \href{https://doi.org/10.1007/978-3-662-03276-3}{doi:10.1007/978-3-662-03276-3}.

\bibitem[LT10]{LT}
J.~M. Landsberg and O.~Tommasi, \emph{On the {Debarre--de Jong} and
  {Beheshti--Starr} conjectures on hypersurfaces with too many lines}, Michigan
  Math. J. \textbf{59} (2010), no.~3, 573--588,
  \href{https://doi.org/10.1307/mmj/1291213957}{doi:10.1307/mmj/1291213957}.

\bibitem[Nob75]{Nobile}
A.~Nobile, \emph{Some properties of the {Nash} blowing-up}, Pacific J. Math.
  \textbf{60} (1975), no.~1, 297--305,
  \href{https://doi.org/10.2140/pjm.1975.60.297}{doi:10.2140/pjm.1975.60.297}.

\bibitem[PRST25]{PRST}
A.~Patel, E.~Riedl, G.~Smith, and D.~Tseng, \emph{Lines highly tangent to a
  hypersurface}, Accepted for publication in Proc. Amer. Math. Soc., 2025.

\bibitem[Stacks]{Stacks}
{The Stacks Project Authors}, \emph{{Stacks Project}},
  \url{https://stacks.math.columbia.edu}.

\end{thebibliography}

\Addresses
\end{document}